\documentclass[a4paper,11pt]{amsart} 
\usepackage{amsfonts, amsmath,amssymb,amsbsy,amstext,amsxtra,latexsym}
\usepackage[all]{xy}

\newcommand{\si}{\sigma}

\newtheorem{thm}{Theorem}[section]

\newtheorem{theorem}[thm]{Theorem}

\newtheorem{lemma}[thm]{Lemma}
\newtheorem{prop}[thm]{Proposition}

\newtheorem{rmk}[thm]{Remark}
\newtheorem{exm}[thm]{Example}

\title{Stability of hypersurface sections of quadric threefolds}
\author{Sangho Byun and Yongnam Lee}
\date{}

\address{Department of Mathematical Sciences, KAIST, 291 Daehak-ro, Yuseong-gu, Daejon 305-701, Korea}

\email{capqus@kaist.ac.kr}

\address{Department of Mathematical Sciences, KAIST, 291 Daehak-ro, Yuseong-gu, Daejon 305-701, Korea}

\email{ynlee@kaist.ac.kr}

\subjclass[2010]{Primary 14L24, Secondary 14J15}

\begin{document}
\maketitle
\begin{abstract} Let $S$ be a complete intersection of a smooth quadric 3-fold $Q$ and a hypersurface of degree $d$ in $\mathbb P^4$. In this paper we analyze GIT stability of $S$ with respect to the natural $G=SO(5, \mathbb C)$-action. We prove that if $d\ge 4$ and $S$ has at worst semi-log canonical singularities then $S$ is $G$-stable. Also,
we prove that if $d\ge 3$ and $S$ has at worst semi-log canonical singularities then $S$ is $G$-semistable.
\end{abstract}

\section{Introduction}
By the geometric invariant theory (GIT) analysis, Gieseker \cite{Gieseker} proved the existence of a quasiprojective coarse moduli space ${\mathcal M}_{K^2, \chi(\mathcal O_S)}$ for smooth projective surfaces $S$ of $K_S$ ample with fixed numerical invariants $K_S^2$ and $\chi(\mathcal O_S)$. In \cite{Gieseker}, he verified that $S$ is asymptotically Hilbert stable. With the result of Mabuchi \cite{Mabuchi}, both asymptotic Chow stability and asymptotic Hilbert stability coincide.

With the proof of bounds for log surfaces with given $K^2$, Alexeev \cite{Alexeev} clarified the construction of projective coarse moduli space of surfaces of general type with fixed $K^2$ that was started by Koll\'ar and Shepherd-Barron \cite{KSB}. The compactified moduli space, which is called KSBA compactification, should include (possibly reducible) surfaces with ordinary double curves and certain other mild singularities. These singularities are semi-log canonical singularities, and log canonical singularities for normal cases. We refer Definition 2.34 in \cite{KM} or Definition 4.17 in \cite{KSB} for these singularities. These surfaces are called smoothable stable surfaces.

However, this compactification is difficult to understand, and there is no description of it even relatively simple examples such as the quintic surfaces. Recently, there has been an approach \cite{Gallardo} to quintic surfaces via geometric invariant theory for describing GIT compactification and for comparing KSBA compactification with GIT compactification.

\medskip

Let $Q$ be the smooth quadric threefold in $\mathbb{P}^4$ defined by the equation
\[ x_0x_4+x_1x_3+x_2^2=0.\]
Since every nonsingular quadric hypersurface in $\mathbb{P}^4$ is projectively equivalent to $Q$, a complete intersection of a smooth quadric and a hypersurface of degree $d(d\ge3)$ can be identified with an element in $|\mathcal{O}_Q(d)|=\mathbb{P}(V)$, where $V$ is a vector space defined by the exact sequence
\[ 0\to H^0(\mathbb{P}^4,\mathcal{O}_{\mathbb{P}^4}(d-2))\to H^0(\mathbb{P}^4,\mathcal{O}_{\mathbb{P}^4}(d))\to V \to 0. \]

The automorphism group of $Q$ is isomorphic to the reductive group $G:=SO(5, \mathbb C)$. Let $S$ be a complete intersection of a quadric 3-fold and a hypersurfaces of degree $d$ in $\mathbb P^4$. The main portion of this paper is devoted to GIT stability analysis of $S$ induced by the $G$-action. Our GIT stability analysis makes us to compare a part of KSBA compactification with GIT compactification. The situation studied by us is special, and it does not help in understanding the general theory. But a comparative study on KSBA compactification and GIT compactification is just started. Moreover, GIT stability analysis of surfaces of general type is almost unknown except the beautiful result of Gieseker \cite{Gieseker}.

\medskip
In this paper, we precisely prove the following two theorems.
\begin{theorem}
Suppose $S$ is a complete intersection of a smooth quadric hypersurface and a hypersurface of degree $d$ in $\mathbb{P}^4$.
Suppose $d\ge 4$ and $S$ has at worst semi-log canonical singularities. Then $S$ is $G$-stable.
\end{theorem}

\begin{theorem}
Suppose $S$ is a complete intersection of a smooth quadric hypersurface and a hypersurface of degree $d$ in $\mathbb{P}^4$.
Suppose $d\ge 3$ and $S$ has at worst semi-log canonical singularities. Then $S$ is $G$-semistable.
\end{theorem}

A similar approach is done in cubic sections of a smooth quadric threefold \cite{LiTian}. Our GIT semistability of $S$ does not imply Chow semistability of $S$ in $\mathbb P^4$ (Example 2.13). If $S$ is a complete intersection defined by hypersurfaces with arbitrary degree in $\mathbb{P}^4$, then GIT stability analysis is hard to describe. By Theorem 1.5 in \cite{Ferretti} (cf. Theorem 1.1 in \cite{Sano}), a complete intersection of two stable (resp. semistable) hypersurfaces is stable (resp. semistable). But the main difficulty arises when one is not stable.

\medskip

We prove our main theorems using GIT stability analysis to understand the type of singularities when it is not stable or unstable. We also remark on strictly semistable points with minimal orbits. We generalize a part of contents in \cite{LiTian}. In this paper, we work on the field of complex numbers.

\bigskip

\section{Proof of Theorems}

Let $Q$ be the smooth quadric threefold in $\mathbb{P}^4$ defined by the equation
\[ x_0x_4+x_1x_3+x_2^2=0.\]
Since every nonsingular quadric hypersurface in $\mathbb{P}^4$ is projectively equivalent to $Q$, a complete intersection of a smooth quadric and a hypersurface of degree $d(d\ge3)$ can be identified with an element in $|\mathcal{O}_Q(d)|=\mathbb{P}(V)$, where $V$ is a vector space defined by the exact sequence
\[ 0\to H^0(\mathbb{P}^4,\mathcal{O}_{\mathbb{P}^4}(d-2))\to H^0(\mathbb{P}^4,\mathcal{O}_{\mathbb{P}^4}(d))\to V \to 0. \]

Take the set of monomials
\[ \mathcal{B}:=\{x_0^{a_0}x_1^{a_1}\dots x_4^{a_4}|\sum_{i=0}^4 a_i=d\,\, \text{and}\,\, a_0a_4=0\} \]
to be a basis of $V$. Since the automorphism group of $Q$ is isomorphic to $SO(5)$,
we can assume that the one parameter subgroups(1-PS) of $SO(5)$ are diagonalized and their weights are normalized to:
\begin{equation*}
\lambda_{u,v}=diag(t^u,t^v,1,t^{-v},t^{-u}) \quad\text{with}\quad u\ge v\ge 0
\end{equation*}
Then the weight of a monomial $x_0^{a_0}x_1^{a_1}\dots x_4^{a_4}\in \mathcal{B}$ with respect to $\lambda_{u,v}$ is
\begin{equation*}
W_{\lambda_{u,v}}(x_0^{a_0}x_1^{a_1}\dots x_4^{a_4})=(a_0-a_4)u+(a_1-a_3)v
\end{equation*}

Let \begin{equation*}
\begin{aligned}
M_{<0}(\lambda_{u,v})=\{x_0^{a_0}x_1^{a_1}\dots x_4^{a_4}\in \mathcal{B}|(a_0-a_4)u+(a_1-a_3)v<0\},\\
M_{\le 0}(\lambda_{u,v})=\{x_0^{a_0}x_1^{a_1}\dots x_4^{a_4}\in \mathcal{B}|(a_0-a_4)u+(a_1-a_3)v\le 0\}.
\end{aligned}
\end{equation*}

Due to the Hilbert-Mumford criterion \cite{Mumford} \cite{GIT}, an element $f(x_0,\dots,x_4)\in \mathbb{P}(V)$ is stable or
semistable if and only if the inequality
$\mu(\lambda, f ) > 0$ or $\mu(\lambda, f )\geq 0$, respectively,
holds for every non-trivial one parameter subgroup $\lambda: \mathbb G_m(\mathbb C)\to G$.
Now, for $\si\in G$, the formula
\[\mu(\si\cdot\lambda\cdot\si^{-1}, \si\cdot f)=\mu(\lambda, f) \]
holds true. This formula can be put $\lambda$ into some normalized form. Then, $f$
is stable or semistable, if
$\mu(\lambda, \si\cdot f) > 0$ or $\mu(\lambda, \si\cdot f) \geq 0$, respectively,
holds for every normalized one parameter subgroup $\lambda$ and every $\si\in G$.

\begin{lemma}\label{lemma2.1}
If $\lambda$ is any normalized 1-PS, then $M_{\le 0}(\lambda)$ is a subset of one of $M_{\le 0}(\lambda_{u,v})$, where $(u,v)=(1,0)$ or $\frac{u}{v}\le d-1$ for $v\neq 0$.
\end{lemma}

\begin{proof}
Consider the case $v=0$, then we note that $M_{\le 0}(\lambda_{u,0})=M_{\le 0}(\lambda_{1,0})$.

Note that $a_0=0$ if $x_0^{a_0}x_1^{a_1}\dots x_4^{a_4} \in M_{\le 0}(\lambda_{u,v})$ with $v\neq 0$ and $\frac{u}{v}> d-1$. In fact, if $a_0\neq 0$ then $a_4=0$ and $a_3\le d-1$. Therefore $0\le a_1\le (d-1)+a_1-a_3 < a_0\frac{u}{v}+a_1-a_3=\frac{1}{v}W_{\lambda_{u,v}}(x_0^{a_0}x_1^{a_1}\dots x_4^{a_4})$.

\medskip
Now let $\lambda_{u,v}$ be a normalized 1-PS with $v\neq 0$ and $\frac{u}{v}> d-1$. Take any $x_1^{a_1}x_2^{a_2}x_3^{a_3}x_4^{a_4} \in M_{\le 0}(\lambda_{u,v})$.

If $a_4=0$, $a_1\le a_3$ and it is in $M_{\le 0}(\lambda_{d-1,1})$.

If $a_4\ge 1$ then
\begin{equation*}
\begin{aligned}
\frac{1}{v}W_{\lambda_{u,v}}(x_1^{a_1}x_2^{a_2}x_3^{a_3}x_4^{a_4})&=-a_4\frac{u}{v}+a_1-a_3\\
&<-a_4(d-1)+a_1-a_3\\
&\le -(d-1)+a_1-a_3\\
&\le -(d-1)+(d-1)-a_3 \quad(\text{because}\, \, \, a_1\le d-1)\\
&\le 0.
\end{aligned}
\end{equation*}
So all $x_1^{a_1}x_2^{a_2}x_3^{a_3}x_4^{a_4}\in \mathcal{B}$ with $a_4\ge 1$ is in $M_{\le 0}(\lambda_{u,v})$. And it is also in $M_{\le 0}(\lambda_{d-1,1})$ because $W_{\lambda_{d-1,1}}(x_1^{a_1}x_2^{a_2}x_3^{a_3}x_4^{a_4})=-a_4(d-1)+a_1-a_3\le 0$.
In all, $M_{\le 0}(\lambda_{u,v})\subset M_{\le 0}(\lambda_{d-1,1})$.
\end{proof}

\begin{lemma}\label{lemma2.2}
If $\lambda$ is any normalized 1-PS, then $M_{<0}(\lambda)$ is a subset of one of $M_{<0}(\lambda_{u,v})$, where  $v\neq 0$ and $\frac{u}{v}<d-1$ or $\frac{u}{v}=d$.
\end{lemma}

\begin{proof}
Note that $a_0=0$ if $x_0^{a_0}x_1^{a_1}\dots x_4^{a_4} \in M_{< 0}(\lambda_{u,v})$ with $\frac{u}{v}\ge d-1$. In fact, if $a_0\neq 0$ then $a_4=0$ and $a_3\le d-1$. Therefore $0\le a_1\le (d-1)+a_1-a_3 \le a_0\frac{u}{v}+a_1-a_3=\frac{1}{v}W_{\lambda_{u,v}}(x_0^{a_0}x_1^{a_1}\dots x_4^{a_4})$.

If $x_0^{a_0}x_1^{a_1}\dots x_4^{a_4}\in M_{<0}(\lambda_{1,0})$, then $a_0 < a_4$. So $a_0=0$ and $a_4\neq 0$. Then
\begin{equation*}
\begin{aligned}
W_{\lambda_{d,1}}(x_1^{a_1}\dots x_4^{a_4})&=-a_4d+a_1-a_3\\
&\le -d+a_1-a_3\\
&\le -d+(d-1)-a_3 \quad(\text{because}\,\,\, a_1\le d-1)\\
&=-a_3-1<0
\end{aligned}
\end{equation*}
Thus $M_{<0}(\lambda_{1,0})\subset M_{<0}(\lambda_{d,1})$.

For $d-1\le \frac{u}{v}\le \frac{u'}{v'}$, we have $M_{<0}(\lambda_{u,v})\subset M_{<0}(\lambda_{u',v'})$ because
\[\begin{array}{c}
\frac{1}{v'}W_{\lambda_{u',v'}}(x_1^{a_1}\dots x_4^{a_4})=-a_4\frac{u'}{v'}+a_1-a_3\\
\le -a_4\frac{u}{v}+a_1-a_3=\frac{1}{v}W_{\lambda_{u,v}}(x_1^{a_1}\dots x_4^{a_4}).
\end{array}\]
Thus $M_{<0}(\lambda_{u,v})\subset M_{<0}(\lambda_{d,1})$ for $d-1\le \frac{u}{v}\le d$.

\medskip
Now we assume that $\lambda_{u,v}$ is a normalized 1-PS with $v\neq 0$ and $\frac{u}{v}> d$.
We take any $x_1^{a_1}x_2^{a_2}x_3^{a_3}x_4^{a_4} \in M_{<0}(\lambda_{u,v})$.

If $a_4=0$, then $a_1<a_3$ and it is in $M_{<0}(\lambda_{d,1})$.

If $a_4\ge 1$ then
\begin{equation*}
\begin{aligned}
\frac{1}{v}W_{\lambda_{u,v}}(x_1^{a_1}x_2^{a_2}x_3^{a_3}x_4^{a_4})&=-a_4\frac{u}{v}+a_1-a_3\\
&<-a_4d+a_1-a_3\\
&\le -d+a_1-a_3\\
&\le -d+(d-1)-a_3 \quad(\text{because}\,\,\,a_1\le d-1)\\
&<0.
\end{aligned}
\end{equation*}
So all $x_0^{a_0}x_1^{a_1}\dots x_4^{a_4}\in \mathcal{B}$ with $a_4\ge 1$ is in $M_{<0}(\lambda_{u,v})$. And it is also in $M_{<0}(\lambda_{d,1})$ because $W_{\lambda_{d,1}}(x_0^{a_0}x_1^{a_1}\dots x_4^{a_4})=-a_4d+a_1-a_3<0$. Thus $M_{< 0}(\lambda_{u,v})\subset M_{< 0}(\lambda_{d,1})$.
\end{proof}

\begin{lemma}\label{lemma2.3}
Let $S=Q\cap Y$ for some degree $d$ hypersurface $Y\subset \mathbb{P}(V)$ defined by $f$.
Suppose $f$ is a general form whose all monomials in $f$ are contained in one of the maximal subsets $M_{\le 0}(\lambda_{u,v})$ in Lemma~\ref{lemma2.1}.

Then $S$ is singular along a line $L:x_2=x_3=x_4=0$ if $\frac{u}{v}<d-1$ for $v\neq 0$ and $S$ has an isolated singularity if $(u,v)=(1,0)$ or $(u,v)=(d-1,1)$.
\end{lemma}

\begin{proof}
Suppose that all monomials in $f$ are contained in some maximal subset $M_{\le 0}(\lambda_{u,v})$ with $v\neq 0$, that is $W_{\lambda_{u,v}}(x_0^{a_0}x_1^{a_1}\dots x_4^{a_4})=(a_0-a_4)u+(a_1-a_3)v\le 0$ for all monomials $x_0^{a_0}x_1^{a_1}\dots x_4^{a_4}$ of $f$.

Suppose $x_0^{a_0}x_1^{a_1}\dots x_4^{a_4}\in M_{\le 0}(\lambda_{u,v})$ with $a_0+a_1\ge d-1$.

If $a_0+a_1=d$, then $a_2=a_3=a_4=0$ and so $W_{\lambda_{u,v}}(x_0^{a_0}x_1^{a_1}\dots x_4^{a_4})=a_0u+a_1v>0$.

If $a_0+a_1=d-1$, then
\begin{align}
&a_3=1 \quad\text{and}\quad a_4=0 \Longrightarrow a_0u+(a_1-1)v\le 0\\
&a_3=0 \quad\text{and}\quad a_4=1 \Longrightarrow (a_0-1)u+a_1v\le 0\\
&a_3=0 \quad\text{and}\quad a_4=0 \Longrightarrow a_0u+a_1v\le 0, \quad\text{we get a contradiction.}
\end{align}

Since $u\ge v\ge 0$, (1) and (2) have the minimum when $a_0=0, a_1=d-1$ and the minimum value is $(d-2)v$ for (1), and $-u+(d-1)v$ for (2). Since $0<(d-2)v$, we must have $d-1 \le \frac{u}{v}$. Therefore, if $(u,v)\neq (d-1,1)$, then every monomial $x_0^{a_0}x_1^{a_1}\dots x_4^{a_4}$ in $f$ has $a_2+a_3+a_4\ge 2$. Hence $S$ is singular along a line $L$.

We can easily compute the followings.
\begin{equation*}
M_{\le 0}(\lambda_{1,0})=\{x_1^{a_1}x_2^{a_2}x_3^{a_3}x_4^{a_4}|\sum_{i=1}^4 a_i=d\}
\end{equation*}
and $M_{\le 0}(\lambda_{d-1,1})$ has monomials with the maximal weight zero:
\[x_0x_3^{d-1},x_1^{d-1}x_4,x_2^d,x_1x_2^{d-2}x_3,\dots,x_1^{c}x_3^{c}\,\, (\text{resp.}\,\, x_1^{c}x_2x_3^{c})\]
where $d=2c$ (resp. $d=2c+1$).
So $S$ has an isolated singularity at $p_0=[1,0,0,0,0]$ if $(u,v)=(1,0)$ or $(u,v)=(d-1,1)$.
\end{proof}

\begin{lemma}\label{lemma2.4}
Let $S=Q\cap Y$ for some degree $d$ hypersurface $Y\subset \mathbb{P}(V)$ defined by $f$.
Suppose $f$ is a general form whose all monomials in $f$ are contained in one of the maximal subsets $M_{<0}(\lambda_{u,v})$ in Lemma~\ref{lemma2.2}.

Then $S$ is singular along a line $L:x_2=x_3=x_4=0$ if $(u,v)\neq (d,1)$ and $S$ has an isolated singularity if $(u,v)=(d,1)$.
\end{lemma}

\begin{proof}
Suppose that all monomials in $f$ are contained in some maximal subset $M_{< 0}(\lambda_{u,v})$. Assume that there is a monomial $x_0^{a_0}x_1^{a_1}\dots x_4^{a_4}$ of $f$ is in $M_{<0}(\lambda_{u,v})$ with $a_0+a_1\ge d-1$.
$W_{\lambda_{u,v}}(x_0^{a_0}x_1^{a_1}\dots x_4^{a_4})=(a_0-a_4)u+(a_1-a_3)v<0$.

If $a_0+a_1=d$, then $a_2=a_3=a_4=0$ and so $W_{\lambda_{u,v}}(x_0^{a_0}x_1^{a_1}\dots x_4^{a_4})=a_0u+a_1v>0$.

If $a_0+a_1=d-1$, then
\begin{align}
&a_3=1 \quad\text{and}\quad a_4=0 \Longrightarrow a_0u+(a_1-1)v<0\\
&a_3=0 \quad\text{and}\quad a_4=1 \Longrightarrow (a_0-1)u+a_1v<0\\
&a_3=0 \quad\text{and}\quad a_4=0 \Longrightarrow a_0u+a_1v<0, \quad\text{contradiction.}
\end{align}

Since $u\ge v\ge 0$, (4) and (5) have the minimum when $a_0=0, a_1=d-1$ and the minimum value is $(d-2)v$ for (4), $-u+(d-1)v$ for (5). Since $0<(d-2)v$, we must have $d-1<\frac{u}{v}$. Therefore, if $(u,v)\neq (d,1)$, then every monomial $x_0^{a_0}x_1^{a_1}\dots x_4^{a_4}$ in $f$ has $a_2+a_3+a_4\ge 2$ and hence $S$ is singular along a line $L$.

One can easily check that $M_{<0}(\lambda_{d,1})$ has monomials with the maximum weight $-1$:
\[x_1^{d-1}x_4,x_2^{d-1}x_3,x_1x_2^{d-3}x_3^2,\dots, x_1^{c-1}x_2x_3^{c}\,\, (\text{resp.} \, \, x_1^{c}x_3^{c+1})\]
where $d=2c$ (resp. $d=2c+1$).
So $S$ has isolated singularity at $p_0=[1,0,0,0,0]$ if $(u,v)=(d,1)$.
\end{proof}

By Lemma 2.3, if $S$ is a general non stable element then $S$ is singular along a line or an isolated singularity. We will show that $S$ is not semi-log canonical. Then by the open condition of semi-log canonical surface singularities, all non stable elements are not semi-log canonical.

\begin{prop}\label{prop1}
Let $deg(Y)=d\ge 4$. Suppose $S$ has singularities along a line and $S$ is not stable.
Then $S$ is not semi-log canonical.
\end{prop}

\begin{proof}
Let $f$ be the equation of $Y$. Suppose $S$ has singularities along a line and $S$ is not stable. By Lemma~\ref{lemma2.3}, all monomials in $\si\cdot f$ for some $\si\in G$ are contained in the maximal subset $M_{\le 0}(\lambda_{u,v})$ with $\frac{u}{v}<d-1$ for $v\neq 0$. Consider the points $p=[a,b,0,0,0]$ on the line $L:x_2=x_3=x_4=0$. We assume that $a\neq 0$.
Choose the affine coordinate $y_i=x_i/x_0$
Then the affine equation near $p'=(0,0,0)$ in $\mathbb{C}^3$ is
\[ 0=\si\cdot f(1,y_1-\frac{b}{a},y_2,y_3,-y_2^2-(y_1-\frac{b}{a})y_3)\]

We will show that the point $p_0=[1,0,0,0,0]$ is not semi-log canonical.
The affine equation near $p_0$ is
\[0=\si\cdot f(1,y_1,y_2,y_3,-y_2^2-y_1y_3)=f_r+f_{r+1}+\dots+f_m\]
where $f_i$ homogeneous in $y_1,y_2,y_3$ of degree $i$ and $f_r\not\equiv 0$.
So the ${\rm mult}_{p_0}(\si\cdot S)=r$. If $x_0^{a_0}x_1^{a_1}\dots x_4^{a_4}$ is a monomial in $\si\cdot f$ such that $a_0$ is the largest, then $r\ge d-a_0$.
So if $d\ge 5$, then $r\ge 3$. In fact, for any monomial in $\si\cdot f$, $a_0<d-2$ when $d\ge 5$ because if $a_0\ge d-2$, then
\begin{equation*}
\begin{aligned}
a_0\frac{u}{v}+a_1-a_3&\ge (d-2)\frac{u}{v}+a_1-a_3\\
&\ge (d-2)+a_1-a_3\\
&\ge (d-2)-2=d-4 \quad(\text{because}\,\, a_3\le 2)\\
&> 0. \quad(\text{because}\,\, d\ge 5)
\end{aligned}
\end{equation*}
and so $W_{\lambda_{u,v}}(x_0^{a_0}x_1^{a_1}\dots x_4^{a_4})=a_0u+(a_1-a_3)v>0$.

Since $p_0$ is a non isolated singularity with multiplicity$\geq 3$, $p_0$ is not semi-log canonical singularity by the classification of the semi-log canonical surface singularities (Theorem 4.24 in \cite{KSB}). Since $\si\cdot S$ is not semi-log canonical, $S$ is neither semi-log canonical.

\medskip
If $d=4$, one can easily check that $M_{\le 0}(\lambda_{1,1})$ is the only maximal subset such that $\frac{u}{v}<3$ and has monomials with $a_0\ge 2$. And the monomial is $x_0^2x_3^2$. So for $\si\cdot f$ in the linear span of $M_{\le 0}(\lambda_{1,1})$, $f_r=y_3^2$. Then by considering terms of degree it is not a pinch point. Again by the classification of the semi-log canonical surface singularities (Theorem 4.24 in \cite{KSB}), it is not semi-log canonical.
\end{proof}

A log canonical singularity can be checked by the computation of log canonical threshold. Log canonical thresholds can be calculated from a set of weights associated with the variables.

\begin{lemma}\label{kollar} {\rm [Proposition 8.14 in \cite{Kollar}]} Let $f$ be a holomorphic function near $0 \in \mathbb C^n$. and $D=\{ f=0\}$. Assign rational weights $w(x_i)$ to
the variables $x_i$, and let $w(f)$ be the weighted multiplicity of $f$. Then $$c_0(f)\le\frac{\sum{w(x_i)}}{w(f)}.$$
And the equality holds if the weighted homogeneous leading term $f_w$ of
$f$ has an isolated critical point at the origin or if
$f_w(x_1^{w(x_1)}, \dots , x_n^{w(x_n)})=0\subset\mathbb P^{n-1}$ is smooth.
\end{lemma}

\begin{prop}\label{prop2}
Let $deg(Y)=d\ge 4$. Suppose $S$ is normal and $S$ is not stable.
Then $S$ is not log canonical.
\end{prop}

\begin{proof}
Let $f$ be the equation of $Y$. Suppose that $S$ is normal and $S$ is not stable. By Lemma~\ref{lemma2.3}, all monomials in $\si\cdot f$ for some $\si\in G$ are contained in the maximal subset $M_{\le 0}(\lambda_{1,0})$ or $M_{\le 0}(\lambda_{d-1,1})$.

Let $(u,v)=(1,0)$. After choosing the affine coordinates as in before, consider the affine equation near $p_0=[1,0,0,0,0]$. Then the ${\rm mult}_{p_0}(\si\cdot S)=d$ because $M_{\le 0}(\lambda_{1,0})=\{x_1^{a_1}x_2^{a_2}x_3^{a_3}x_4^{a_4}|\sum_{i=1}^4 a_i=d\}$. So it is not log canonical.

Now let $(u,v)=(d-1,1)$. Since $M_{\le 0}(\lambda_{d-1,1})$ has monomials with the maximal weight zero $x_0x_3^{d-1},x_1^{d-1}x_4,x_2^d, x_1x_2^{d-2}x_3,\dots,x_1^{c}x_3^{c}$ (resp. $x_1^{c}x_2x_3^{c}$) where $d=2c$ (resp. $d=2c+1$), $mult_{p_0}(S)\ge d-1$.

When $d=4$, a general form of $\si\cdot f$ is
\[x_0x_3^3+x_1^2x_3^2+x_4g_3(x_1,x_2,x_3,x_4)+x_1x_3g_2(x_2,x_3)+g_4(x_2,x_3,x_4)\]
where $g_i$ is a polynomial of degree $i$. Choose the affine coordinates as in before. Then the affine equation near $p_0=[1,0,0,0,0]$ is
\begin{equation*}
\begin{aligned}
0&=\si\cdot f(1,y_1,y_2,y_3,-y_2^2-y_1y_3)\\
&=y_3^3+y_1^2y_3^2+(-y_2^2-y_1y_3)g_3(y_1,y_2,y_3,-y_2^2-y_1y_3)+y_1y_3g_2(y_2,y_3)\\
&\, \, \, +g_4(y_2,y_3,-y_2^2-y_1y_3).
\end{aligned}
\end{equation*}

Assign a weight $w=(2,3,4)$, then the log canonical threshold of $\si\cdot f$ at $p_0$
\begin{equation*}
c_0(f)\le \frac{9}{12}=\frac{3}{4}.
\end{equation*}
So $\si\cdot S$ is not log canonical, neither is $S$.
\end{proof}

By Lemma 2.4, if $S$ is a general unstable element then $S$ is singular along a line or an isolated singularity. We will show that $S$ is not semi-log canonical when $S$ has singularities along the line. Then by the open condition of semi-log canonical surface singularities, all unstable elements are not semi-log canonical.

\begin{prop}\label{prop3}
Let $deg(Y)=d\ge 3$. Suppose $S$ has singularities along a line and $S$ is unstable.
Then $S$ is not semi-log canonical.
\end{prop}

\begin{proof}
Let $f$ be the equation of $Y$. Suppose $S$ has singularities along a line and $S$ is unstable.
By Lemma~\ref{lemma2.4}, all monomials in $\si\cdot f$ for some $\si\in G$ are contained in the maximal subset $M_{<0}(\lambda_{u,v})$ with $\frac{u}{v}<d-1$ for $v\neq 0$. Consider the points $p=[a,b,0,0,0]$ on the line $L:x_2=x_3=x_4=0$. We assume that $a\neq 0$. Choose the affine coordinate as in before.
Then the affine equation near $p'=(0,0,0)$ in $\mathbb{C}^3$ is
\[ 0=\si\cdot f(1,y_1-\frac{b}{a},y_2,y_3,-y_2^2-(y_1-\frac{b}{a})y_3)\]

We will show that the point $p_0=[1,0,0,0,0]$ is not semi-log canonical.
The affine equation near $p_0$ is
\[0=\si\cdot f(1,y_1,y_2,y_3,-y_2^2-y_1y_3)=f_r+f_{r+1}+\dots+f_m\]
where $f_i$ homogeneous in $y_1,y_2,y_3$ of degree $i$ and $f_r\not\equiv 0$.
And if $x_0^{a_0}x_1^{a_1}\dots x_4^{a_4}$ is a monomial in $\si\cdot f$ such that $a_0$ is the largest, then $r\ge d-a_0$.
So if $d\ge 4$, then $r\ge 3$. In fact, for any monomial in $\si\cdot f$, $a_0<d-2$ when $d\ge 4$ because if $a_0\ge d-2$, then
\begin{equation*}
\begin{aligned}
a_0\frac{u}{v}+a_1-a_3&\ge (d-2)\frac{u}{v}+a_1-a_3\\
&\ge (d-2)+a_1-a_3\\
&\ge (d-2)-2=d-4 \quad(\text{because}\,\, a_3\le 2)\\
&\ge 0. \quad(\text{because}\,\, d\ge 4)
\end{aligned}
\end{equation*}
and so $W_{\lambda_{u,v}}(x_0^{a_0}x_1^{a_1}\dots x_4^{a_4})=a_0u+(a_1-a_3)v\ge 0$.

Since $p_0$ is a non isolated singularity with multiplicity$\geq 3$, $p_0$ is not semi-log canonical singularity by the classification of the semi-log canonical surface singularities (Theorem 4.24 in \cite{KSB}). Since $\si\cdot S$ is not semi-log canonical, $S$ is neither semi-log canonical.

\medskip
Now, we consider the case $d=3$. Suppose that $M_{<0}(\lambda_{u,v})$ is a maximal subset with $\frac{u}{v}<2$.  One can check easily that if there is a monomial with $a_0\neq 0$ in $M_{<0}(\lambda_{u,v})$ then the monomial is $x_0x_3^2$. So for $\si\cdot f$ in the linear span of $M_{<0}(\lambda_{u,v})$, $f_r=y_3^2$. Then by considering terms of degree it is not a pinch point. Again by the classification of the semi-log canonical surface singularities (Theorem 4.24 in \cite{KSB}), it is not semi-log canonical.
\end{proof}

By Proposition~\ref{prop1} and Proposition~\ref{prop2}, we get the following theorem.

\begin{theorem}\label{stable}
Suppose $S$ is a complete intersection of smooth quadric hypersurface and a hypersurface of degree $d$ in $\mathbb{P}^4$.
Suppose $d\ge 4$ and $S$ has at worst semi-log canonical singularities. Then $S$ is stable.
\end{theorem}

\begin{theorem}
Suppose $S$ is a complete intersection of a smooth quadric hypersurface and a hypersurface of degree $d$ in $\mathbb{P}^4$.
Suppose $d\ge 3$ and $S$ has at worst semi-log canonical singularities. Then $S$ is semistable.
\end{theorem}

\begin{proof}
By Theorem~\ref{stable}, it is sufficient to show that when $d=3$, and if all monomials of $\si\cdot f$ for some $\si\in G$ are in $M_{<0}(\lambda_{3,1})$, then $\si\cdot S$ is not log canonical.
One can check easily that monomials with the maximal weight in $M_{<0}(\lambda_{3,1})$ are $x_1^2x_4, x_1x_3^2, x_2^2x_3$ and a general form of $\si\cdot f$ is \[x_4q(x_1,x_2,x_3,x_4)+x_3^2l(x_1,x_2,x_3,x_4)+x_3q'(x_2,x_3,x_4)\]
where $q,q'$ are quadratic polynomials and $l$ is a linear polynomial.
Choose an affine coordinate as in before, then the equation near $p_0=[1,0,0,0,0]$ is
\[\begin{array}{l}
(-y_2^2-y_1y_3)q(y_1,y_2,y_3,-y_2^2-y_1y_3)+y_3^2l(y_1,y_2,y_3,-y_2^2-y_1y_3)\\
+y_3q'(y_2,y_3,-y_2^2-y_1y_3)=0.
\end{array}\]

Assign a weight $w=(2,3,4)$, then
\begin{equation*}
c_0(f)\le \frac{9}{10}.
\end{equation*}
So $\si\cdot S$ is not log canonical, and it implies $S$ is not log canonical.
\end{proof}

\medskip

Before finishing this section, we remark on strictly semistable points with minimal orbits. For $f\in \mathbb P(V)$ which is not properly stable, using the special 1-PS $\lambda_{u,v}$, the limit $lim_{t\to \infty}f_t=f_0$ exists and it is invariant with respect to $\lambda_{u,v}$. The invariant part of polynomials in any maximal subset $M_{\le 0}(\lambda_{u,v})$ have a common specialization, which we denote by Type $(\xi)$:
\begin{equation*}
\begin{aligned}
&\mu_0 x_2^d+\mu_1 x_1x_2^{d-2}x_3+\dots+\mu_c x_1^{c}x_3^{c}=0 \quad\text{if}\quad d=2c\\
&\mu_0 x_2^d+\mu_1 x_1x_2^{d-2}x_3+\dots+\mu_c x_1^{c}x_2x_3^{c}=0 \quad\text{if}\quad d=2c+1.
\end{aligned}
\end{equation*}

If $S$ is of Type $(\xi)$, it is strictly semistable with closed orbits due to Luna's criterion.

\begin{lemma}\label{luna} {\rm (Luna's criterion \cite{Luna})} Let $G$ be a reductive group acting on an affine variety $V$. If $H$ is a reductive subgroup of $G$ and $x\in V$ is stabilized by $H$, then the orbit $G\cdot x$ is closed in $V$ if and only if $C_G(H)\cdot x$ is closed in $V^H$ where $C_G(H)$ is the centralizer and $V^H$ is the fixed point set.
\end{lemma}

\begin{prop}
If $S$ is of Type $(\xi)$, it is strictly semistable with closed orbits.
\end{prop}

\begin{proof}
The stabilizer of Type $(\xi)$ contains a 1-PS:
\begin{equation*}
H=\{diag(t^{d-1},t,1,t^{-1},t^{-d+1)})|t\in \mathbb C^*\},
\end{equation*}
of distinct weights. The semi-stability is obtained by using the Kempf-Morrison criterion (Proposition 2.4 in \cite{AFS}). And the centralizer
\begin{equation*}
C_G(H)=\{diag(a_0,a_1,1,a_1^{-1},a_0^{-1})\}\subset SO(5)
\end{equation*}
is a maximal torus. It acts on the fixed point set
\[V^H=\langle x_0x_3^{d-1},x_1^{d-1}x_4,x_2^d, x_1x_2^{d-2}x_3,\dots,x_1^{c}x_3^{c}\,\, ({\rm resp}.\, x_1^{c}x_2x_3^{c}) \rangle \subset V\]
where $d=2c$ (resp. $d=2c+1$). It is straightforward to see any element of Type $(\xi)$ is semistable with closed orbit in $V^H$ under the action. Then the proof follows from Luna's criterion.
\end{proof}

Suppose $f(x_1, x_2, x_3)$ is of Type $(\xi)$ and $g(x_0, x_1, x_2, x_3, x_4)$ is in $M_{<0}(\lambda_{u,v})$ for some 1-PS $\lambda_{u, v}$. Then $S=Q\cap Y$, where the equation of $Y$ is
\[f(x_1, x_2, x_3)+g(x_0, x_1, x_2, x_3, x_4),\]
is strictly semistable but the orbit is not closed. It degenerates to Type $(\xi)$.

\medskip

\begin{exm}
Let $S=Q\cap Y$ in $\mathbb{P}^4$ such that $Y$ is defined by the equation
\[ x_0x_3^3 +x_1x_2^2x_3=0. \]
Then $S$ is semistable by the Kempf-Morrison criterion (Proposition 2.4 in \cite{AFS}) and Proposition 2.12.

Now let $\chi(t)=diag(t^3,t^3,t^{-2},t^{-2},t^{-2})$ be a 1-PS of $SL(5)$. Then $\mu(Q,\chi)=\max\{3-2,3-2,-4\}=1$ and $\mu(Y,\chi)=\max\{ 3-6, 3-4-2\}= -3$. By Theorem 1.5 in \cite{Ferretti} (cf. Theorem 1.1 in \cite{Sano}), $\mu(S,\chi)=deg(Y)\mu(Q,\chi)+deg(Q)\mu(Y,\chi)=-2<0$. So $S$ is Chow unstable.

\end{exm}

\begin{rmk}
Consider the same surface $S$ in Example 2.12. $S$ is semistable. But $S$ has singularities along the line $x_2=x_3=x_4=0$ whose general point on the line is not normal crossing. So $S$ is not sem-log canonical by the classification of the semi-log canonical surface singularities (Theorem 4.24 in \cite{KSB}).
\end{rmk}

\medskip

{\em Acknowledgements}. The authors thank the referee for pointing out several mistakes in the original version of the paper. This work was supported by Basic Science Research Program through the National Research Foundation of Korea funded by the Korea government(MSIP)(No.2013006431). Also the second named author was supported by the National Research Foundation of Korea funded by the Korea government(MSIP)(No.2013042157).

\bigskip

\begin{small}\end{small}

\end{document}